\documentclass{amsart}
\usepackage{amsmath}
\usepackage{amssymb}
\usepackage{amsthm}
\usepackage{tikz}
\usepackage[utf8]{inputenc}
\usepackage{pgfplots}
\usepackage{wasysym}
\usepackage{color}
\usepackage{ulem}
\usepackage{hyperref}

\usetikzlibrary{decorations.markings}
\usetikzlibrary{arrows.meta}

\newtheorem{theorem}{Theorem}
\newtheorem{corollary}[theorem]{Corollary}
\newtheorem{lemma}[theorem]{Lemma}
\newtheorem{defn}[theorem]{Definition}

\title[Linking number of monotonic cycles in random book embeddings]{Linking number of monotonic cycles in random book embeddings of
complete graphs}
\author[Aguillon]{Yasmin Aguillon}
\address{Department of Mathematics,
University of Notre Dame,
255 Hurley Bldg,
Notre Dame, IN 46556 USA}
\email{yaguillo@nd.edu}

\author[Burkholder]{Eric Burkholder}
\address{Department of Mathematics,
University of Kentucky,
719 Patterson Office Tower,
Lexington, KY 40506 USA}
\email{ebu241@uky.edu}

\author[Cheng]{Xingyu Cheng}
\address{Department of Mathematics,
University of North Carolina at Chapel Hill,
120 E Cameron Avenue,
Chapel Hill, NC 27599 USA}
\email{xcheng1@unc.edu}

\author[Eddins]{Spencer Eddins}
\address{University of Kentucky,
Lexington, KY 40506}
\email{spencer.eddins@uky.edu}

\author[Harrell]{Emma Harrell}
\address{Mount Holyoke College,
50 College Street,
South Hadley, MA 01075 USA}
\email{harre22e@mtholyoke.edu}

\author[Kozai]{Kenji Kozai}
\address{Department of Natural Sciences and Mathematics,
Lesley University,
29 Everett Street,
Cambridge, MA 02138 USA}
\email{kkozai@lesley.edu}

\author[Leake]{Elijah Leake}
\address{DePaul University,
1 E. Jackson Boulevard,
Chicago, IL 60604 USA}
\email{ejohnj247@gmail.com}

\author[Morales]{Pedro Morales}
\address{Department of Mathematics,
Purdue University,
150 N. University Street,
West Lafayette, IN 47907 USA}
\email{moralep@purdue.edu}



\thanks{The authors were supported in part by NSF Grant DMS-1852132.}

\begin{document}

\begin{abstract}
A book embedding of a complete graph is a spatial embedding
whose planar projection has the vertices located along a circle, consecutive
vertices are connected by arcs of the circle, and the projections of the
remaining ``interior'' edges in the graph are straight line segments between
the points on the circle representing the appropriate vertices. A random
embedding of  a complete graph can be generated by randomly assigning relative
heights to these interior edges. We study a family of two-component links that
arise as the realizations of pairs of disjoint cycles in these random embeddings
of graphs. In particular, we show that the distribution of linking numbers can
be described in terms of Eulerian numbers. Consequently, the mean of the squared
linking number over all random embeddings is $\frac{i}{6}$, where $i$ is the
number of interior edges in the cycles. We also show that the mean of the
squared linking number over all pairs of $n$-cycles in $K_{2n}$ grows linearly
in $n$.
\end{abstract}

\maketitle

\section{Introduction}

Random knot models have been used to study the spatial configurations of
polymers such as DNA, whose length is 1,000 to 500,000 times the length of the
diameter of the nucleus \cite{Flapan}. With such a long molecule confined to
a compact space, DNA can become knottted, tangled, or
linked. In order for cell replication to occur, DNA must unknot
itself with the aid of a special enzyme known as topoisomarase that cuts through
the knotted parts of the DNA molecule and reconnects any loose ends, and
problems can arise during cellular replication if topoisomarase enzymes do not work properly \cite{Beals}. By comparing the topological invariants of DNA
before and after enzymes act on it, we can learn more about
mechanisms of these enzymes and their effects on the structure of DNA
\cite{Mishra}. Because many polymers are too small to image in detail,
several authors have used mathematical models to study configurations of
long polymer chains by introducing versions of uniform random distributions
of polygonal chains in a cube
\cite{Arsuaga09,Arsuaga,Diao,Diao93,panagiotou10,Portillo,Numerical}.
Even-Zohar, et al. introduced a random model based on petal diagrams of knots
and links where the distribution of links can be studied in terms of random
permutations, achieving an explicit description of the asymptotic distribution
for the linking number \cite{Even_Zohar_2016}.

Random graph embeddings can be thought of as generalizations of random
knot embeddings to molecules with non-linear structures. In \cite{flapan16},
a random graph embedding model generalizing the uniform random distributions
of polygonal chains in a cube was used study the behavior of linking numbers
and writhe. In this paper, we study an alternate random embedding model
similar to the Petaluma model in \cite{Even_Zohar_2016} in that the distribution
of random embeddings can be
described in terms of a random choice of permutations. This model is based
on book embeddings of the complete graph $K_n$. Rowland has classified all
possible links that could appear in book embeddings of $K_6$ \cite{Rowland},
and we consider the more general case of links in $K_{2n}$. In particular, we
study a special class of two-component links that appear in book embedding
which are unions of disjoint monotonic cycles, and we describe the behavior of
the linking number in terms of the combinatorial properties of the length of
the cycles and the number of interior edges in the book embedding. We show that
the mean value of the squared linking number grows linearly with respect to both
quantitites in Theorem \ref{thm:meansquaredlinking} and Theorem
\ref{thm:totallinking}.

\section{Random book embeddings}

Given a graph $G$, Atneosen \cite{atneosen72} and Persinger \cite{persinger66} 
introduced the notion of a \textit{book embedding} of $G$, which is a
particular class of spatial embedding of a graph  in which the vertices of the
graph are placed along a fixed line in $\mathbb{R}^3$ called the
\textit{spine} of the book. The edges of $G$ are embedded on half-planes,
called \textit{sheets}, which are bounded by the spine.  Classically, the
edges are drawn as disjoint circular arcs on their respective sheets. Instead,
we will consider the \textit{circular diagram} for a book embedding of $K_n$
introduced by Endo and Otsuki in which the spine is a circle consisting of the
vertices and edges between consecutive
vertices, the pages are discs bounded by the spine, and the remaining edges are
straight lines between vertices of a given page  \cite{endo94, endo96}. 

We focus on book embeddings of the complete graph $K_{2n}$ (or sometimes
$K_{m+n}$) on  $2n$ verticles. In our model, the $2n$ vertices will be labeled
as $v_1, \dots, v_{2n}$ in clockwise order around the circular spine. The
perimeter of the circle will form the edges between consecutive
vertices $v_j$ and $v_{j+1}$ for all $j \in \{ 1, 2, \cdots , 2n \}$, where
the indices are taken modulo $2n$. We denote these edges as
\textit{exterior edges}. The remaining $\binom{2n}{2} -2n$ edges are
\textit{interior edges}, and a book embedding is determined by dividing the
interior edges among a finite number of sheets so that no two edges within a
page intersect.

In order to generate a \textit{random book embedding}, we embed each interior
edge on its own separate sheet. The ordering of sheets can then be determined
by a random permutation $\sigma$ of $\{1,\dots,\binom{2n}{2}-2n\}$ with the
uniform distribution. We can think
of the permutation as giving the height order of the sheets, so that edge $e_i$
is in a sheet above edge $e_j$ if $\sigma(i)>\sigma(j)$. Note that a random
book embedding will typically be equivalent to a book embedding with far
fewer sheets. When edges in two adjacent sheets do not cross in a circular
diagram, the two sheets can be combined to a single sheet in which
the two edges are embedded without intersecting, obtaining an equivalent
embedding with one fewer sheet.

\section{Preliminary definitions}

The image of two disjoint cycles in a graph $G$ under an embedding forms a
two-component link. We can compute the linking number of any oriented link $L$
in $\mathbb{R}^3$ by considering the signed crossings of the two components in
a planar projection with the rule indicated in Figure
\ref{fig:positive-and-negative-crossings}. We will denote half of the sum of
the signed crossings as the linking number $\ell(L)$ of a link $L$.
This gives a quantitative measure of how interwined the two components are. In
an abuse of notation, given two oriented cycles $P$ and $Q$ of a graph $G$
and a fixed embedding, we will let $\ell(P \cup Q)$ mean the linking number
of the image of the two cycles under the embedding.

\begin{figure}
\centering
\begin{tikzpicture}
\draw[black, thick, -stealth] (-1,-1) -- (1,1);
\draw[black, thick] (1,-1) -- (0.1,-0.1);
\draw[black, thick,-stealth] (-0.1,0.1) -- (-1,1);
\end{tikzpicture}
\hspace{25mm}
\begin{tikzpicture}
\draw[black, thick,] (-1,-1) -- (-0.1,-0.1);
\draw[black, thick,-stealth] (1,-1) -- (-1,1);
\draw[black, thick,-stealth] (0.1,0.1) -- (1,1);
\end{tikzpicture}
\caption{A positive crossing (left) and a negative crossing (right)}
\label{fig:positive-and-negative-crossings}
\end{figure}
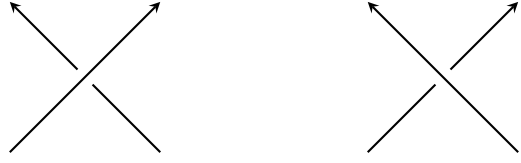

We introduce a special class of links in book embeddings of a graph. 

\begin{defn}
	Let $K_{2n}$ be a complete graph with vertices enumerated as
	$\{v_1,\dots,v_{2n}\}$ in cyclic order along the spine of a book
	embedding of $K_{2n}$. An oriented cycle with consecutive edges
	$\{\overrightarrow{v_{i_1}v_{i_2}}, \overrightarrow{v_{i_2}v_{i_3}},
	\dots, \overrightarrow{v_{i_{k-1}}v_{i_k}}, \overrightarrow{v_{i_k}v_{i_1}}
	\}$ is 

\begin{enumerate}
\item \textit{strictly increasing} if there is a cyclic permutation
$i_1',\dots,i_k'$ of
$i_1,\dots,i_k$ such that $i_j' < i_{j+1}'$ for all $j \in \{ 1,2,\dots,
k-1 \}$.
\item \textit{strictly decreasing} if there is a cyclic permutation
$i_1',\dots,i_k'$ of
$i_1,\dots,i_k$ such that $i_j' > i_{j+1}'$ for all $j \in \{ 1,2,\dots,
k-1 \}$.
\item \textit{monotonic} if the cycle is either strictly increasing or
strictly decreasing.
\end{enumerate}
\end{defn}

The 4-cycle on the left in Figure \ref{fig:monotonicdefinition} is
monotonic because beginning with the vertex $v_1$, the vertices in the cycle
in order are $v_1, v_2, v_3, v_4$, which has strictly increasing indices.
However, the order of the vertices in the 4-cycle on the right is
$v_1,v_3,v_2,v_4$. The indices are not monotonic even up to cyclic permutation,
so this cycle is not monotonic.

\begin{figure} \centering \begin{tikzpicture} \draw[fill=black] (1,1) circle
(2.25pt); \draw[fill=black] (-1,1) circle (2.25pt); \draw[fill=black] (1,-1)
circle (2.25pt); \draw[fill=black] (-1,-1) circle (2.25pt);
    
	\draw[thick] (1,1) -- (1,-1); \draw[thick] (1,-1) -- (-1,-1); \draw[thick,]
	(-1,-1) -- (-1,1); \draw[thick] (-1,1) -- (1,1);
    
	\node[] at (0,1) ($>$){$>$};
    
	\node[] at (1.3,1.3) (1){$v_1$};
	\node[] at (1.3,-1.3) (2){$v_2$};
	\node[] at (-1.3,-1.3) (3){$v_3$};
	\node[] at (-1.3,1.3) (4){$v_4$};
    
	\node[] at (0,-1.75) (1234){$v_1v_2v_3v_4$};
	\node[] at (0,-2.2) {monotonic};
	\draw[fill=black] (6,1) circle (2.25pt); \draw[fill=black] (4,1) circle
	(2.25pt); \draw[fill=black] (6,-1) circle (2.25pt); \draw[fill=black]
	(4,-1) circle (2.25pt);
    
	\draw[thick] (6,1) -- (4,-1); \draw[thick] (6,-1) -- (4,1); \draw[thick,]
	(4,-1) -- (6,-1); \draw[thick] (4,1) -- (6,1);
    
	\node[] at (5,1) ($>$){$>$};
    
	\node[] at (6.3,1.3) (1){$v_1$}; \node[] at (6.3,-1.3) (2){$v_2$};
	\node[] at (3.7,-1.3) (3){$v_3$}; \node[] at (3.7,1.3) (4){$v_4$};
    
	\node[] at (5,-1.75) (1324){$v_1v_3v_2v_4$}; \node[] at (5,-2.2) {non-monotonic};
	\end{tikzpicture} \caption{Monotonic (left) and non-monotonic (right)
	cycles} \label{fig:monotonicdefinition} \end{figure}
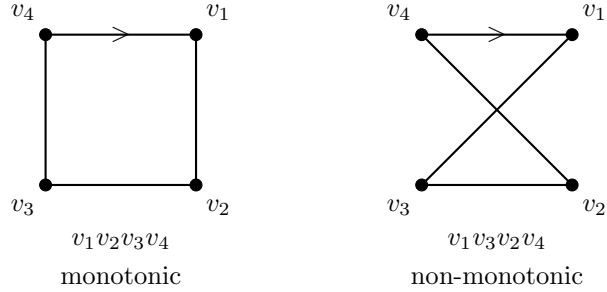

Finally, we also introduce the Eulerian numbers, which arise in combinatorics
as coefficients of Eulerian polynomials \cite{comtet,euler00,a008292}.

\begin{defn}
	Let $\sigma \in S_n$ be a permutation on $\{1,\dots,n\}$. An
	\textit{ascent} of the permutation is a value $1\leq k \leq n-1$ such
	that $\sigma(k) < \sigma(k+1)$.
\end{defn}

\begin{defn}
	The \textit{Eulerian number} $A(n,m)$ is the number of permutations
	$\sigma \in S_n$ that have exactly $m$ ascents.
\end{defn}

As an example, we have the following exhaustive list of permutations in
$S_3$:

\begin{center}
	(1,2,3);
	(1,3,2);
	(2,1,3);
	(2,3,1);
	(3,1,2);
	(3,2,1).
\end{center}

Among these permutations, (1,2,3) has two ascents, (1,3,2), (2,1,3), (2,3,1),
and (3,1,2) each have one ascent, and (3,2,1) has no ascents. Hence,
$A(3,2)=1$, $A(3,1) = 4$, and $A(3,0) = 1$. Note that $A(n,n)=0$ for all
$n>0$. Additionally, there is always exactly one permutation in $S_n$ with no
ascents and exactly one permutation in $S_n$ with $n-1$ descents, which are
($n$,$n-1$,\dots,1) and (1,2,\dots,$n$), respectively. Hence,
$A(n,0)=A(n,n-1)=1$.

Eulerian numbers are coefficients of Eulerian polynomials,
\begin{equation*}
	A_n(t) = \sum_{m=0}^n A(n,m)t^m,
\end{equation*}
where $A_n(t)$ is recursively defined by the relations,
\begin{align*}
	A_0(t)&=1,\\
	A_n(t)&=t(1-t)A'_{n-1}(t)+A_{n-1}(t)(1+(n-1)t), &\text{ for }n >0.
\end{align*}
It is also known that
\begin{equation*}
	A(n,m) = \sum_{k=0}^{m+1}(-1)^k \binom{n+1}{k}(m+1-k)^n,
\end{equation*}
and the exponential generating function for the Eulerian numbers is
\begin{equation*}
	\sum_{n=0}^\infty \sum_{m=0}^\infty A(n,m) t^m \frac{x^n}{n!} =
		\frac{t-1}{t-e^{(t-1)x}}.
\end{equation*}
	
From the definition, it is also evident that for a fixed $n$, the sum of
Eulerian numbers $A(n,m)$ over all possible values of $m$ gives the number
of all permutations, $|S_n|$, so that
\begin{equation*}
	\sum_{m=0}^n A(n,m) = n!.
\end{equation*}

\section{Linking numbers of disjoint monotonic cycles}

In this paper, we will consider the distribution of linking numbers of two
disjoint monotonic cycles in random book embeddings.
First, note the following fact about the number of interior edges of two
monotonic cycles in a book embedding.

\begin{lemma}\label{lem:interior-edges}
Two disjoint monotonic cycles of length
$m$ and $n$ in a book embedding of $K_{m+n}$ must have an equal number of
interior edges, which is also equal to half the number of crossings between the
two cycles.
\end{lemma}

\begin{proof}
Let $P$ and $Q$ be an $m$-cycle and $n$-cycle in a book embedding,
respectively, and suppose that $P$ has $i$ interior edges. Let
$\overrightarrow{v_jv_k}$ be an interior edge of $P$. Then $v_{k-1}$ must be a
vertex in $Q$, and there is a smallest $h>k$ such that $v_h$ is a
vertex in $Q$. Then $\overrightarrow{v_{k-1}v_h}$ is an edge in $Q$ which
crosses the edge $\overrightarrow{v_jv_k}$ of $P$. Similarly, there is an edge
$\overrightarrow{v_sv_{j+1}}$ in $Q$ that crosses $\overrightarrow{v_jv_k}$,
and no other edge in $Q$ can cross $\overrightarrow{v_jv_k}$. Hence, the number
of crossings between $P$ and $Q$ is twice the number of interior edges in $P$.
By symmetry, this is also equal to twice the number of interior edges in $Q$.
\end{proof}

Lemma \ref{lem:interior-edges} implies that if $P$ and $Q$ are both
$n$-cycles and $P$ consists of $n$ interior edges, then all edges in $Q$ must
also be interior.
We now relate the number of disjoint cycles with fixed linking number to the
Eulerian numbers $A(m,n)$.

\begin{theorem}\label{thm:eulerstriangledistribution}
Suppose $P$ and $Q$ are both strictly increasing $n$-cycles in $K_{2n}$ so
that $P$ and $Q$ both consist of $n$ interior edges. The proportion of random
book embeddings of $K_{2n}$ for which $P$ and $Q$ have linking number 
equal to $\ell$ is
\begin{equation*}
	\frac{A(2n-1,n+\ell-1)}{(2n-1)!}.
\end{equation*}
\end{theorem}

\begin{proof}
Let $P$ and $Q$ be two strictly increasing cycles, each with $n$
interior edges. Consider a permutation of all of the interior edges of
$K_{2n}$, which determines the ordering of their respective sheets in a
book embedding. As we are only concerned with the linking number
$\ell(P \cup Q)$, we only need the relative orderings of the
edges of $P$ and $Q$ in order to resolve the signs of any crossings between
interior edges of $P$ and $Q$. By designating these edges as $e_1,\dots,
e_{2n}$, we may consider the permutation $\sigma$ as a permutation of
$\{1,\dots,2n\}$.

Without loss of generality, we label the topmost edge of the permutation of
interior edges as edge $e_{2n}$. Since the edges in the cycle are directed so
that the cycle is strictly increasing, we may begin numbering the vertices of
$K_{2n}$ so that the initial vertex of $e_{2n}$ is vertex $v_{2n}$. We then
number the vertices in cyclic order, so that the vertex in $K_{2n}$ that lies
next in the clockwise direction from $v_{2n}$ is $v_1$, the following vertex
(which is the terminal vertex of $e_{2n}$) is $v_2$, and so on. The edge
indices will then also be identified with their initial vertex, so that
the edge $\overrightarrow{v_1v_3}$ is $e_1$, the edge $\overrightarrow{v_2v_4}$
is $e_2$, and so on, until the edge $\overrightarrow{v_{2n-1}v_1}$ is
labeled $e_{2n-1}$ and edge $\overrightarrow{v_{2n}v_2}$ is labeled $e_{2n}$.
Under this labeled scheme, edge $e_j$ will have crossings with edges $e_{j-1}$
and $e_{j+1}$, where indices are taken modulo $2n$.

The bijective function $\sigma$ from $\{1,\dots,2n\}$ to itself determines the
relative heights of the edges so that
whenever $\sigma(j) > \sigma(k)$, then $e_j$ is in a sheet above the sheet
containing $e_k$, and whenever $\sigma(j) < \sigma(k)$, $e_j$ is embedded in a
sheet below the sheet containing $e_k$. Since both cycles are strictly
increasing, the sign of the crossing between edge $e_j$ and edge $e_{j+1}$ can
be determined by $\sigma(j)$ and $\sigma(j+1)$. When $\sigma(j)>\sigma(j+1)$,
the sign of the crossing is negative. When $\sigma(j)<\sigma(j+1)$, the sign of
the crossing is positive, as seen in Figure \ref{fig:pi(j)pi(j+1)crossings}.
Therefore, the linking number is half the quantity of the number of times
$\sigma(j) < \sigma (j+1)$ minus the number of times $\sigma(j) > \sigma(j+1)$.

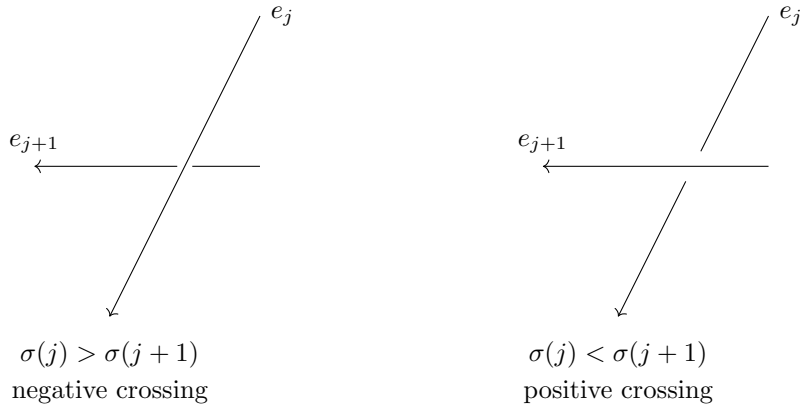
\begin{figure}
\centering
\begin{tikzpicture}
\draw[] (2,0) -- (1.1,0);
\draw[->] (.9,0) -- (-1,0);
\draw[->] (2,2) -- (0,-2);
\node[] at (-1,.3){$e_{j+1}$};
\node[] at (2.3,2){$e_j$};
\node[] at (0,-2.5){$\sigma(j)>\sigma(j+1)$};
\node[] at (0,-3){negative crossing};
\end{tikzpicture}
\hspace{25mm}
\begin{tikzpicture}
\draw[->] (2,0) -- (-1,0);
\draw[] (2,2) -- (1.1,.2);
\draw[->] (.9,-.2) -- (0,-2);
\node[] at (-1,.3){$e_{j+1}$};
\node[] at (2.3,2){$e_j$};
\node[] at (0,-2.5) {$\sigma(j)<\sigma(j+1)$};
\node[] at (0,-3){positive crossing};
\end{tikzpicture}
\caption{A negative crossing (left) and a positive crossing (right) in terms
of $\sigma(j)$ and $\sigma(j+1)$}
\label{fig:pi(j)pi(j+1)crossings}
\end{figure}

By construction, $\sigma(2n)=2n$, so that $\sigma(2n-1) < \sigma(2n)$ and
$\sigma(2n)> \sigma(1)$. Since this results in exactly one positive crossing and
one negative crossing, crossings involving the edge $e_{2n}$ have zero net
effect on the linking number. We may ignore edge $2n$ in the permutation and
consider only a further restriction of the permutation to a permutation
$\sigma^\prime$ of $\{1,\dots,2n-1\}$. Topologically,
this can be thought of as applying a Reidemeister Move 2, sliding the topmost
edge away to the exterior of the binding so that the edge $e_{2n}$ no longer
has any crossings with edges $e_{2n-1}$ and $e_1$

Notice that $\sigma^{\prime} (j) < \sigma^{\prime}(j+1)$ is the same as an
ascent in $\sigma^{\prime}$ and $\sigma^{\prime} (j) > \sigma^{\prime} (j+1)$
is the same as a descent in $\sigma^{\prime}$. So the linking number of $P$
and $Q$ depends on the number of ascents of the permutation $\sigma^{\prime}$.
If $\sigma^\prime$ has $m$ ascents, it has $2n-2-m$ descents,
so that the linking number is $\frac{1}{2}[m-(2n-2-m)]$. Setting this equal
to $\ell$, then $m=n+\ell-1$. Thus, we conclude that the number of
permutations in $S_{2n-1}$ that lead to a linking number of $\ell$ is
$A(2n-1,n+\ell-1)$. For each permutation $\sigma^\prime \in
S_{2n-1}$, there are an equal number of permutations of the edges of
$K_{2n}$ that restrict to $ \sigma^\prime$, so that the proportion of random
book embeddings in which $P$ and $Q$ have linking number $\ell$ is
\begin{equation*}
	\frac{A(2n-1,n+\ell-1)}{(2n-1)!}.
\end{equation*}

\end{proof}

\begin{figure} \centering
\begin{tikzpicture}
	\draw[fill=black] (0,0) circle (2pt);
	\draw[fill=black] (-3.5,2) circle (2pt);
	\draw[fill=black] (3.5,2) circle (2pt);
	\draw[fill=black] (-3.5,4) circle (2pt);
	\draw[fill=black] (3.5,4) circle (2pt);
	\draw[fill=black] (0,6) circle (2pt);

	\node[] at (0,-0.25){$v_4$};
	\node[] at (3.75,2){$v_3$};
	\node[] at (3.75,4){$v_2$};
	\node[] at (0,6.25){$v_1$};
	\node[] at (-3.75,4){$v_6$};
	\node[] at (-3.75,2){$v_5$};

	\draw (1.8375, 3.9)--(3.5,2);
	\draw (3.5,4)--(2.7125,3.1);
	\draw (3.5,2)--(1.85,2);
	\draw (-1.8375,2.1)--(-3.5,4);
	\draw (-3.5,2)--(-2.7125,2.9);
	\draw (-2.5375,3.1)--(-1.8375,3.9);

	\begin{scope}[decoration={
        markings,
        mark=at position 0.5 with {\arrow{Latex[length=3mm, width=3mm]}}}
        ] 
        \draw[postaction={decorate}] (-3.5,4)--(3.5,4);
        \draw[postaction={decorate}] (0,6)--(1.6625,4.1);
        \draw[postaction={decorate}] (-0,0)--(-1.6625, 1.9);
        \draw[postaction={decorate}] (-1.6625,4.1)--(-0,6);
    \end{scope}
    
    \begin{scope}[decoration={
        markings,
        mark=at position 0.7 with {\arrow{Latex[length=3mm, width=3mm]}}}
        ] 
        \draw[postaction={decorate}] (2.5375,2.9)--(0,0);
    \end{scope}
    
    \begin{scope}[decoration={
        markings,
        mark=at position 0.43 with {\arrow{Latex[length=3mm, width=3mm]}}}
        ] 
        \draw[postaction={decorate}] (1.6,2)--(-3.5,2);
    \end{scope}
    
    \node[] at (0,4.4){$\mathbf{e_6}$};
    \node[] at (1.6,5.2){$\mathbf{e_1}$};
    \node[] at (1.6,0.8){$\mathbf{e_2}$};
    \node[] at (0,2.4){$\mathbf{e_3}$};
    \node[] at (-1.6,0.8){$\mathbf{e_4}$};
    \node[] at (-1.6,5.2){$\mathbf{e_5}$};
	
\end{tikzpicture}
\caption{Solomon's link as a union of two monotonic $3$-cycles in
$K_6$.}
\label{fig:ascentlinkingnumberexample} \end{figure}
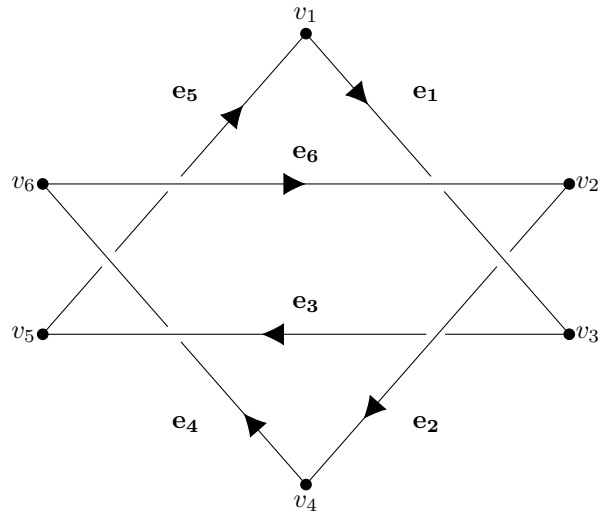

\begin{table} \centering \begin{tabular}{c|c|c|c} $j$ & $\sigma(j)$ & crossing of $e_j$ and $e_{j+1}$ & ascent or descent \\
\hline 1 & 5 & $-$ & descent\\
2 & 4 & $-$ & descent\\
3 & 3 & $-$ & descent\\
4 & 2 & $-$ & descent\\
5 & 1 & $+$ & ascent\\
6 & 6 & $-$ & \\
\end{tabular}
\caption{Signed crossings and ascents/descents in height function $\sigma$
for the example in Figure \ref{fig:ascentlinkingnumberexample}.}
\label{tab:euleriantrianglek6}
\end{table}

An example of the connection between ascents, descents, crossing signs, and
linking number is shown in Figure \ref{fig:ascentlinkingnumberexample} and
Table \ref{tab:euleriantrianglek6}. Observe
in Table \ref{tab:euleriantrianglek6} that $\sigma(5)<\sigma(6)$. Thus
$j=5$ would be an ascent. However, as $\sigma(6)>\sigma(1)$, the signed crossing
between $e_5$ and $e_6$ is canceled out with the signed crossing between
$e_6$ and $e_1$. Considering only $j=1$ ,$2$, $3$, $4$ we are left with 
four descents, which lead to four negative crossings and a linking number of
$-2$.

We remark that the results from Theorem \ref{thm:eulerstriangledistribution}
extend to the more general case of two monotonic cycles of length $m$ and $n$
with $i$ interior edges each. The sign of the linking number will flip whenever
we reverse the orientation of one of the cycles, so if we have two monotonic
cycles $P$ and $Q$ of length $n$ which are not necessarily strictly increasing,
this would result in replacing $\ell$ with $-\ell$ in the result of Theorem
\ref{thm:eulerstriangledistribution}. However, the Eulerian numbers have the
symmetry property that $A(n,m) = A(n,n-1-m)$, so that $A(2n-1,n-\ell-1)=
A(2n-1,n+\ell-1)$. This results in an identical proportion of book embeddings
in which the cycles have linking number $\ell$, thus whether the cycles are
strictly increasing or strictly decreasing has no net effect on the
distribution of linking numbers as long as they are both monotonic.

In the case where $P$ and $Q$ have lengths $m$ and $n$, respectively, Lemma
\ref{lem:interior-edges} states that both $P$ and $Q$ have the same number
of interior edges, which we will denote by $i$. Contracting $K_{m+n}$ along
all of the exterior edges in $P$ and $Q$ does not alter the topological type
of the link $P \cup Q$, and the proportion of random book embeddings of
$K_{m+n}$ for which the linking number of $P\cup Q$ is equal to $\ell$ will
be the same as the proportion of book embeddings of the contracted graph $K'$
in which the linking number of $P \cup Q$ is equal to $\ell$ by a similar
argument as in Theorem \ref{thm:eulerstriangledistribution}. Hence, we arrive
at the following when $i \geq 3$.

\begin{corollary} \label{cor:strictly-increasing-to-monotonic}
Let $P$ and $Q$ be monotonic cycles of length $m$ and $n$, respectively, in
$K_{m+n}$. The proportion of random book embeddings of $K_{m+n}$ in which
the linking number of $P \cup Q$ is equal to $\ell$ is
\begin{equation*}
	\frac{A(2i-1,i+\ell-1)}{(2i-1)!},
\end{equation*}
where $i \geq 2$ is the number of interior edges of both $P$ and $Q$.
\end{corollary}

The exceptional case when $i=2$ can be verified to follow the same formula
as in Corollary \ref{cor:strictly-increasing-to-monotonic} by contracting to
two $3$-cycles with two interior edges and one exterior edge each, then
applying the argument in Theorem \ref{thm:eulerstriangledistribution} to the
interior edges only. Table \ref{tab:Eulerianlinkingnumberdistributionbycrossing}
gives the values
of $A(2i-1,i+\ell-1)$ for $1 \leq i \leq 5$. The proportion of
random book embeddings for which two cycles with $i$ interior edges have
a linking number of $\ell$ can be obtained by dividing the entries by
$(2i-1)!$.

\begin{table}
    \centering
    \scriptsize
    \begin{tabular}{c|ccccccccccc}
        $i \backslash \ell$ &-5 &-4 &-3 &-2 &-1 &0 &1 &2 &3 &4 &5 \\
        \hline
        1 & & & & & &1 & & & & & \\
        
        2 & & & & &1 &4 &1 & & & & \\
        
        3 & & & &1 &26 &66 &26 &1 & & & \\
        
        4 & & &1 &120 &1191 &2416 &1191 &120 &1 & & \\
        
        5 & &1 &502 &14608 &88234 &156190 &88234 &14608 &502 &1 & \\
        
    \end{tabular}
    \caption{Values of $A(2i-1,i+\ell-1)$}
    \label{tab:Eulerianlinkingnumberdistributionbycrossing}
\end{table}

The following theorem describes the number of disjoint $m$- and
$n$-cycles with a given number of interior edges. In combination with the
previous corollary, this will allow for calculation of the frequency with which
a random $m$-cycle $P$ and disjoint $n$-cycle $Q$ has linking
number $\ell$ in a random book embedding of $K_{m+n}$.

\begin{theorem}\label{thm:cycle-interior-edge-distribution}
Let $m,n \geq 3$. Then the number of disjoint (undirected) monotonic cycles
$P$ and $Q$ in a
book embedding of $K_{m+n}$ so that $P$ is an $m$-cycle and $Q$ is a $n$-cycle,
each with $2 \leq i \leq \min \{m,n\}$ interior edges is
\begin{equation*}
\binom{m}{m-i} \binom{n-1}{n-i} + \binom{n}{n-i} \binom{m-1}{m-i},
\end{equation*}
if $m \neq n$. In the case that $m=n$, the number of disjoint cycles is
\begin{equation*}
\binom{n}{n-i}\binom{n-1}{n-i}.
\end{equation*}
\end{theorem}

\begin{proof}
Fix a labeling of the vertices of $K_{m+n}$ in cyclic order $v_1,\dots,v_{m+n}$.
Suppose $P$ is a $m$-cycle and $Q$ is a $n$-cycle.

First, suppose $P$ contains $v_1$. If $P$ has $i$ interior edges, there are
$\binom{m}{i}$ ways to choose which of the $m$ edges in $P$ are interior edges.
For each of the $i$ chosen edges in $P$, in order for it to be interior, there
must be a vertex in the cycle $Q$ lying between the initial and terminal
vertices of the edge in $P$. Moreover, for each of the external edges in the
cycle $P$, there cannot be any vertices of $Q$ lying between the initial and
terminal vertices. This create $i$ areas in which the vertices of $Q$ must
be located, one between the initial and terminal vertices of each internal
edge in $P$, with each containing at least one vertex. A stars and bars
argument, in which there are $n-i$ vertices of $Q$ to allocate after placing
one vertex of $Q$ into each of the $i$ spots, and $i-1$ bars to separate the
$i$ spots, leads to $\binom{n-1}{n-i}$ ways of choosing the vertices of $Q$.
This results in $\binom{m}{m-i} \binom{n-1}{n-i}$ choices of $P$ and $Q$ so
that $P$ contains $v_1$ and both cycles have $i$ interior edges.

By an analogous argument, there are $\binom{n}{n-i} \binom{m-1}{m-i}$ ways to
choose $P$ and $Q$ so that $Q$ contains $v_1$, completing the proof when
$m \neq n$.

If $m=n$, there is no distinction between the cases when $v_1$ is in $P$
and $v_1$ is in $Q$.
\end{proof}

The number of disjoint $n$ cycles
in $K_{2n}$ with $i$ interior edges is tabulated in Table
\ref{tab:interior-edges-number-distribution-for-2ngons} for $3 \leq n \leq 10$.

The values $\binom{n}{n-i} \binom{n-1}{n-i}$ appear as OEIS sequence A103371
\cite{a103371} up to a shift in indices due to the cyclic symmetry in the
circular diagrams of book embeddings. The sum over all $i$ gives the number of
ways to choose two disjoint monotonic $n$-cycles in $K_{2n}$. An undirected
monotonic cycle is determined by the vertices in the cycles, so this amounts
to choosing two disjoint subsets of $n$ vertices from the $2n$ vertices in
$K_{2n}$. The number of ways in which this choice can be made is given by
$\binom{2n-1}{n-1}=\binom{2n-1}{n}$.

\begin{table}
    \centering
    \begin{tabular}{c|cccccccccc}
        $n$ $\backslash $ $i$ &1 &2 &3 &4 &5 &6 &7 &8 &9 &10 \\
        \hline
        
        3 &3 &6 &1 & & & & & & & \\
        
        4 &4 &18 &12 &1 & & & & & & \\
        
        5 &5 &40 &60 &20 &1 & & & & & \\
        
        6 &6 &75 &200 &150 &30 &1 & & & & \\
        
        7 &7 &126 &525 &700 &315 &42 &1 & & & \\
        
        8 &8 &196 &1176 &2450 &1960 &588 &56 &1 & & \\
        
        9 &9 &288 &2352 &7056 &8820 &4704 &1008 &72 &1 & \\
        
        10 &10 &405 &4320 &17640 &31752 &26460 &10080 &1620 &90 &1 \\
        \end{tabular}
    \caption{Number of pairs of monotonic $n$-cycles each with $i$ interior
		edges in $K_{2n}$.}
    \label{tab:interior-edges-number-distribution-for-2ngons}
\end{table}

Combining Theorem \ref{thm:cycle-interior-edge-distribution} with Theorem
\ref{thm:eulerstriangledistribution} yields the following corollary.

\begin{corollary}\label{cor:linking-distribution} The proportion of links
$P \cup Q$ with linking number $\ell$ among pairs of $n$-cycles $P$ and
$Q$ in a random book embedding of $K_{2n}$ is
$$\displaystyle
\frac{\displaystyle \sum_{i=1}^{n} \frac{A(2i-1, \ell + i - 1)}{(2i-1)!}
\binom{n}{n-i} \binom{n-1}{n-i}}{\displaystyle \binom{2n-1}{n-1}}.$$
\end{corollary}

\begin{figure} \centering \begin{tikzpicture} \begin{axis}[ xlabel=Linking
Number, ylabel=Proportion of links, legend
style={at={(0.05,.85)},anchor=west} ]

		\addplot[thick,color=cyan] coordinates { (-2, 0.0008333333333)
		(-1, 0.1216666) (0, 0.755) (1,0.12166666) (2,0.0008333333) };

		\addplot[thick,color=red] coordinates { (-3, 0.00000566893424)
		(-2, 0.003537414966) (-1, 0.1667517007) (0, 0.6594104308) (1, 0.1667517007)
		(2, 0.003537414966) (3,0.00000566893424) };

		\addplot[thick,color=blue] coordinates { (-4, 0.00000007873519778)
		(-3, 0.00004247326503) (-2, 0.008067033398) (-1, 0.1955238603) (0,
		0.5927332224) (1,0.1955238603) (2,0.008067033398) (3,0.00004247326503) (4,
		0.00000007873519778) };

		\addplot[thick,color=green] coordinates { (-5, 0) (-4,
		0.0000002893464293) (-3, 0.0001625262039) (-2, 0.01438998832) (-1,
		0.2138922072) (0, 0.5437472306) (1, 0.2138922072) (2, 0.01438998832) (3,
		0.0001625262039) (4,0.0000002893464293) (5,0) };

		\legend{$K_6$, $K_8$, $K_{10}$, $K_{12}$} \end{axis} \end{tikzpicture}
		\caption{Proportion of disjoint pairs of $n$-cycles with a given linking
		number in a random book embedding of $K_{2n}$.}
		\label{fig:combinedeulerianandtriangledistributiongraph} \end{figure}
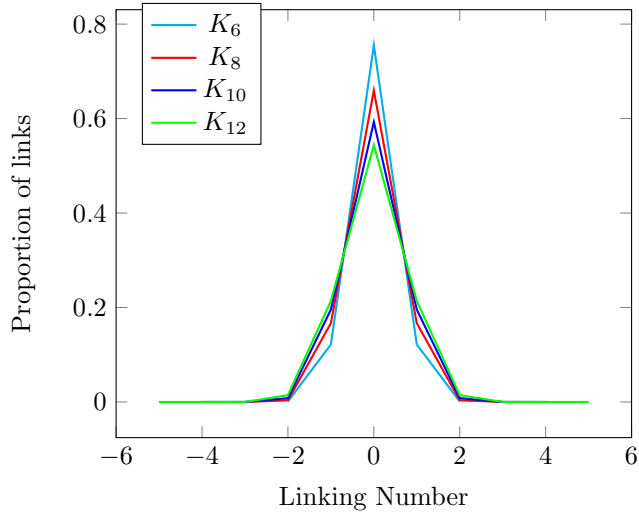

The values from Corollary \ref{cor:linking-distribution} for $n=3$, $4$, $5$,
and $6$ are computed and illustrated in
Figure \ref{fig:combinedeulerianandtriangledistributiongraph}. 
Notice that for two $n$-cycles in $K_{2n}$, the maximum number of crossings
that can appear is $2n$, meaning that an upper bound for the absolute value of
the linking number is $n$. Thus, we can normalize the linking number of two
monotonic cycles by dividing by $n$. The distribution of links with a given
normalized linking number when $n=100$, $200$, $500$, and $1000$,
are shown in Figure \ref{fig:normalizedlinking}. As $n$ increases, the
proportion of links with linking number 0 decreases. However, this behavior
is misleading as links are distributed among a larger range of possible
values for the linking number as $n$ increases. Normalizing the graph to a
density plot as in Figure \ref{fig:normalizedlinkingdensity} gives a very
different picture of the behavior of linking numbers of disjoint $n$-cycles
in random book embeddings of $K_{2n}$. As the number of vertices increases, the
normalized linking numbers tend closer to $0$ as $n$ increases. This model
behaves differently from other models where the mean squared linking number
grows as $\theta(n^2)$, as in \cite{Arsuaga09, Arsuaga, panagiotou10}).

\begin{figure} \centering

	\begin{tikzpicture} \begin{axis} xlabel=Normalized Linking Number,
	ylabel=Proportion, \addplot +[mark=none] table [col sep=comma]
	{full-even-100.csv}; \addplot +[mark=none] table [col sep=comma]
	{full-even-200.csv}; \addplot +[mark=none] table [col sep=comma]
	{full-even-500.csv}; \addplot +[mark=none] table [col sep=comma]
	{full-even-1000.csv}; \legend{$K_{100}$, $K_{200}$, $K_{500}$, $K_{1000}$}
	\end{axis} \end{tikzpicture}

	\caption{Proportion of links with specified  normalized linking number for
	two monotonic $n$-cycles in a random book embedding of $K_{2n}$}
	\label{fig:normalizedlinking} \end{figure}
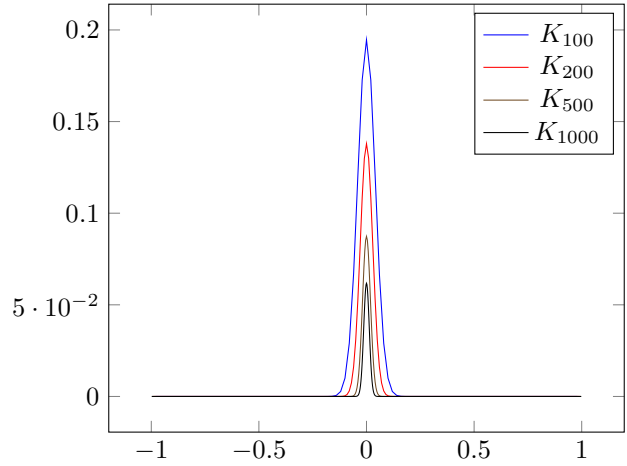

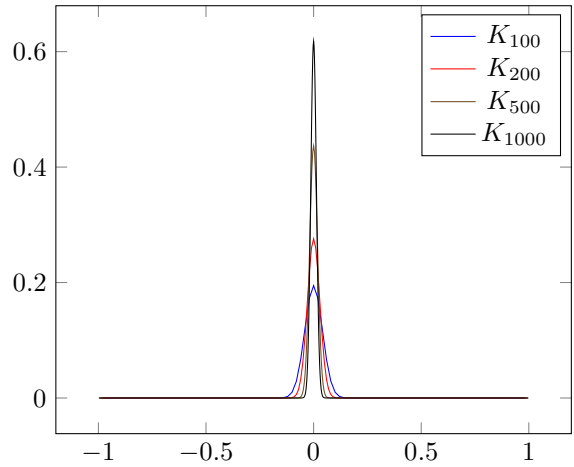
\begin{figure} \centering

	\begin{tikzpicture} \begin{axis} xlabel=Normalized Linking Number,
	ylabel=Density, \addplot +[mark=none] table [col sep=comma]
	{full-even-100-density.csv}; \addplot +[mark=none] table [col sep=comma]
	{full-even-200-density.csv}; \addplot +[mark=none] table [col sep=comma]
	{full-even-500-density.csv}; \addplot +[mark=none] table [col sep=comma]
	{full-even-1000-density.csv}; \legend{$K_{100}$, $K_{200}$, $K_{500}$,
	$K_{1000}$} \end{axis} \end{tikzpicture}

	\caption{Density of links with specified normalized linking number for two
	monotonic $n$-cycles
	in a random book embedding of $K_{2n}$} \label{fig:normalizedlinkingdensity} \end{figure}

In fact, using the exponential generating function for the Eulerian numbers,
we can determine an explicit formula for the mean squared linking number
in terms of the number of interior edges $i$. We will need the following
fact from differential calculus.

\begin{lemma} \label{lem:derivative}
	Let $g(x)=\frac{x^n}{(1-x)^m}$. Then for $k\geq 1$,
	$g^{(k)}(0)=k! \binom{k-n+m-1}{m-1}$.
\end{lemma}

\begin{proof}
	For $|x|<1$, we can express $\frac{1}{1-x}$ as the power series
	\begin{equation*}
		\frac{1}{1-x} = x^0 + x^1 + x^2 + x^3 + \dots.
	\end{equation*}
	Then,
	\begin{equation*}
		g(x) = x^n(x^0+x^1+x^2+x^3+\dots)^m,
	\end{equation*}
	so that $\frac{g^{(k)}(0)}{k!}$ is the coefficient of $x^k$ in the
	power series expansion of $g(x)$. This is the  $x^{k-n}$ coefficient
	of $(x^0+x^1+x^2+x^3+\dots)^m$, which is the number of ways to choose
	$m$ non-negative integers that add up to $k-n$. A stars and bars argument
	counts this as $\binom{k-n+m-1}{m-1}$, with this binomial coefficient
	defined to be $0$ if $k<n$.
\end{proof}

We are now ready to show that the mean squared linking number of
two disjoint cycles grows linearly in the number of interior edges $i$.
Heurestically, this means that we expect that the linking number grows
roughly as the square root of the number of internal edges.

\begin{theorem}\label{thm:meansquaredlinking}
	Let $P\cup Q$ be a union of disjoint $n$ cycles with $i$ interior
	edges each. Then the mean squared linking number of
	$P \cup Q$ in a random book embedding is $\frac{i}{6}$.
\end{theorem}

\begin{proof}
	The exponential generating function for the Eulerian numbers is
	\begin{equation*}
		\sum_{n=0}^\infty \sum_{m=0}^\infty A(n,m) t^m \frac{x^n}{n!} =
			\frac{t-1}{t-e^{(t-1)x}}.
	\end{equation*}
	Multiplying both sides by $t^{-i+1}$, we arrive at,
	\begin{equation*}
		\sum_{n=0}^\infty \sum_{m=0}^\infty A(n,m) t^{m-i+1} \frac{x^n}{n!} =
			\frac{t^{-i+1}(t-1)}{t-e^{(t-1)x}}.
	\end{equation*}
	Notice that differentiating the left-hand side twice with respect to $t$ and taking
	the limit as $t \rightarrow 1$ yields
	\begin{equation*}
		\sum_{n=0}^\infty \sum_{m=0}^\infty \left( (m-i+1)^2-(m-i+1)  \right)
			A(n,m) \frac{x^n}{n!}.
	\end{equation*}
	Differentiating this expression $2i-1$ times with respect to $x$ and
	evaluating at $x=0$ results in
	\begin{equation*}
		\sum_{m=0}^\infty (m-i+1)^2 A(2i-1,m) - (m-i+1)A(2i-1,m).
	\end{equation*}
	After a substitution of $\ell = m-i+1$, this becomes
	\begin{align*}
		\sum_{\ell=-i+1}^{i-1} A(2i-1,i+\ell-1) \ell^2 - A(2i-1,i+\ell-1) \ell
			&=\sum_{\ell=-i+1}^{i-1}A(2i-1,i+\ell-1) \ell^2\\
			&=(2i-1)!E[ \ell(P\cup Q)^2],
	\end{align*}
	as the symmetry in the Eulerian triangle means that the expected value
	of the linking number is 0. Hence, the second part of the summation
	vanishes.

	We now repeat the differentiation on the exponential generating function to
	find an equivalent expression utilizing logarithmic differentiation. We
	set $f(t,x)$ to be the exponential generating function,
	\begin{equation*}
		f(t,x)=\frac{t^{-i+1}(t-1)}{t-e^{(t-1)x}},
	\end{equation*}
	and first compute using L'H\^opital's rule,
	\begin{equation*}
		\lim_{t \rightarrow 1} f(t,x)= 1 \cdot
			\lim_{t\rightarrow 1} \frac{t-1}{t-e^{(t-1)x}}
			= \lim_{t \rightarrow 1} \frac{1}{1-xe^{(t-1)x}} = \frac{1}{1-x}.
	\end{equation*}
	Using logarithmic differentiation, we find that,
	\begin{align*}
		\frac{f_t(t,x)}{f(t,x)} &= \frac{-i+1}{t} + \frac{1}{t-1} -
			\frac{1-xe^{(t-1)x}}{t-e^{(t-1)x}}\\
			&= \frac{-i+1}{t} + \frac{(t-e^{(t-1)x})-(t-1)(1-xe^{(t-1)x})}{
				(t-1)(t-e^{(t-1)x})}\\
			&= \frac{-i+1}{t} + \frac{1-e^{(t-1)x} + (t-1)xe^{(t-1)x}}{(t-1)
				(t-e^{(t-1)x})}.
	\end{align*}
	Taking the limit as $t \rightarrow 1$ using L'H\^opital's rule twice, we
	obtain,
	\begin{align*}
		\lim_{t \rightarrow 1} \frac{f_t(t,x)}{f(t,x)} &=
			(-i+1) + \lim_{t \rightarrow 1} \frac{(t-1)x^2 e^{(t-1)x}}{
			(t-e^{(t-1)x}) + (t-1)(1-xe^{(t-1)x})}\\
		&= (-i+1) + \lim_{t \rightarrow 1} \frac{x^2e^{(t-1)x}+(t-1)x^3e^{(t-1)x}}{
			1-xe^{(t-1)x}+1-xe^{(t-1)x} + (t-1)(-x^2e^{(t-1)x})}\\
		&= (-i+1)+\frac{x^2}{2} \cdot \frac{1}{1-x}.
	\end{align*}
	The second derivative of $\log f(t,x)$ is
	\begin{align*}
		\frac{f_{tt}(t,x)}{f(t)} &- \left(\frac{f_t(t,x)}{f(t,x)}\right)^2
			= -\frac{-i+1}{t^2} - \frac{1}{(t-1)^2} +
				\frac{x^2e^{(t-1)x}}{t-e^{(t-1)x}} +
				\frac{(1-xe^{(t-1)x})^2}{(t-e^{(t-1)x})^2}\\
			&= -\frac{-i+1}{t^2} +
				\frac{-(t-e^{(t-1)x})^2+(t-1)^2[(t-e^{(t-1)x})x^2e^{(t-1)x}
					+(1-xe^{(t-1)x})^2]}{(t-1)^2(t-e^{(t-1)x})^2}.
	\end{align*}
	Taking the limit as $t \rightarrow 1$ using L'H\^opital's rule four times
	yields,
	\begin{equation*}
		\lim_{t \rightarrow 1} \frac{f_{tt}(t,x)}{f(t)} -
			\left(\frac{f_t(t,x)}{f(t,x)}\right)^2
			= -(-i+1) + \frac{x^3}{3} \cdot \frac{1}{(1-x)^2} - \frac{x^4}{12}
				\cdot \frac{1}{(1-x)^2}.
	\end{equation*}
	We can then find,
	\begin{align*}
		\lim_{t \rightarrow 1} f_{tt}(t,x) &=
			\lim_{t \rightarrow 1} f(t)\left(
				\frac{f_{tt}(t,x)}{f(t)}-\left(\frac{f_t(t,x)}{f(t,x)}\right)^2
				+ \left(\frac{f_t(t,x)}{f(t,x)}\right)^2\right)\\
			&= \frac{i(i-1)}{1-x} + \frac{(-i+1)x^2}{(1-x)^2}
				+ \left(\frac{x^3}{3} + \frac{x^4}{6}\right)
				\frac{1}{(1-x)^3}.
	\end{align*}
	By Lemma \ref{lem:derivative}, the $(2i-1)$-th derivative in $x$
	evaluated at $x=0$ is
	\begin{align*}
		(2i-1)!&\left(i(i-1)+(-i+1)(2i-2)+\frac{1}{3}\binom{2i-2}{2}
			+\frac{1}{6}\binom{2i-3}{2}\right)\\
			&= (2i-1)!\left((i-1)(-i+2)+\frac{(2i-2)(2i-3)}{6}
				+ \frac{(2i-3)(2i-4)}{12}\right)\\
			&=(2i-1)!\frac{i}{6}.
	\end{align*}
	Hence,
	\begin{equation*}
		(2i-1)!E[\ell(P\cup Q)^2] = (2i-1)! \frac{i}{6},
	\end{equation*}
	completing the proof of the theorem.
\end{proof}

Using Theorem \ref{thm:meansquaredlinking}, we can find the asymptotic behavior
of the mean squared linking number over all pairs of disjoint $n$ cycles
in $K_{2n}$. Recall that a function $f(n)$ is in order $\theta(n)$ if
there are positive constants $a$, $A$, and $N$ such that
$an \leq f(n) \leq An$ for all $n> N$.

\begin{theorem}\label{thm:totallinking}
	Let $n\geq 3$. Then the mean squared linking number of two cycles
	$P$ and $Q$ taken over all pairs of disjoint $n$-cycles across all
	random book embeddings of $K_{2n}$ is in order $\theta(n)$.
\end{theorem}

\begin{proof}
	By combining Theorem \ref{thm:cycle-interior-edge-distribution}
	and Theorem \ref{thm:meansquaredlinking} and summing over
	the number of interior edges, the mean squared
	linking number is
	\begin{equation*}
		\frac{1}{\binom{2n-1}{n-1}} \sum_{i=2}^n
			\binom{n}{n-i} \binom{n-1}{n-i} \frac{i}{6}.
	\end{equation*}
	Since
	\begin{equation*}
		i\binom{n}{n-i}=i \binom{n}{i}=n \binom{n-1}{i-1},
	\end{equation*}
	this becomes
	\begin{equation}
		\frac{1}{\binom{2n-1}{n-1}} \sum_{i=2}^n
			\frac{n}{6} \binom{n-1}{i-1}^2
		= \frac{n}{6} \cdot \frac{1}{\binom{2n-1}{n-1}} \sum_{i=2}^n
			\binom{n-1}{i-1}^2.
		\label{eqn:summation}
	\end{equation}
	Using Vandermonde's identity, the summation part of the right-hand side
	becomes
	\begin{equation*}
		\sum_{i=2}^n \binom{n-1}{i-1}^2
			= \left(\sum_{i=0}^{n-1} \binom{n-1}{i}^2\right) - \binom{n-1}{0}^2
			= \binom{2n-2}{n-1} - 1.
	\end{equation*}
	Thus, Equation \eqref{eqn:summation} yields
	\begin{equation*}
		\frac{n}{6} \cdot \frac{1}{\binom{2n-1}{n-1}}
			\left( \binom{2n-2}{n-1} - 1 \right)
		= \frac{n}{6} \left( \frac{n}{2n-1} -
			\frac{1}{\binom{2n-1}{n-1}}\right).
	\end{equation*}
	For an upper bound, we have
	\begin{equation*}
		\frac{n}{6} \left( \frac{n}{2n-1} -
			\frac{1}{\binom{2n-1}{n-1}}\right)
			\leq \frac{n}{6}\cdot  \frac{n}{2n-1}
			\leq \frac{n}{6}.
	\end{equation*}
	For a lower bound, we note that if $n\geq 3$,
	\begin{equation*}
		\binom{2n-1}{n-1}= \frac{2n-1}{1} \cdot \frac{2n-2}{2} \cdot
			\cdots \cdot \frac{n+1}{n-1} \cdot \frac{ n}{n}
		\geq (2n-1) (n-1) \geq 2(2n-1).
	\end{equation*}
	Hence,
	\begin{equation*}
		\frac{n}{6} \left( \frac{n}{2n-1} -
			\frac{1}{\binom{2n-1}{n-1}}\right)
		\geq \frac{n}{6} \left( \frac{n}{2n-1} -
			\frac{1}{2(2n-1)}\right)
		= \frac{n}{6} \cdot \frac{n-\frac{1}{2}}{2n-1}
		= \frac{n}{6} \cdot \frac{1}{2} = \frac{n}{12}.
	\end{equation*}
\end{proof}

Sample calculations of the mean squared linking number of two $n$-cycles in
$K_{2n}$ can be seen to asymptotically approach $\frac{n}{12}$, as seen from
the nearly linearly relationship between $n$ and the mean squared linking
number in Figure \ref{fig:meansquaredlinking}. When $n=100$ and $n=1000$, the
approximate value of the mean squared linking number can be computed from the
summation formula in Theorem \ref{thm:totallinking} to be $\approx 8.37521$
and $\approx 83.375$, respectively.

\begin{figure} \centering \begin{tikzpicture} \begin{axis}[ xlabel=$n$, ylabel=Mean squared linking number ]

		\addplot[thick,color=black] coordinates { (3,1/4) (4,38/105) (5,115/252)
		(6,251/462) (7,497/792) (8,13724/19305) (9,2271/2860) (10,243095/277134)
		(11,677435/705432) (12,705431/676039) (13,2343601/2080120) 
		(20,35345263799/20676979323) (30,150336332497705195/59132290782430712)};

		\end{axis} \end{tikzpicture}
		\caption{Mean squared linking number of two disjoint $n$-cycles
		in a random book embedding of $K_{2n}$}
		\label{fig:meansquaredlinking} \end{figure}
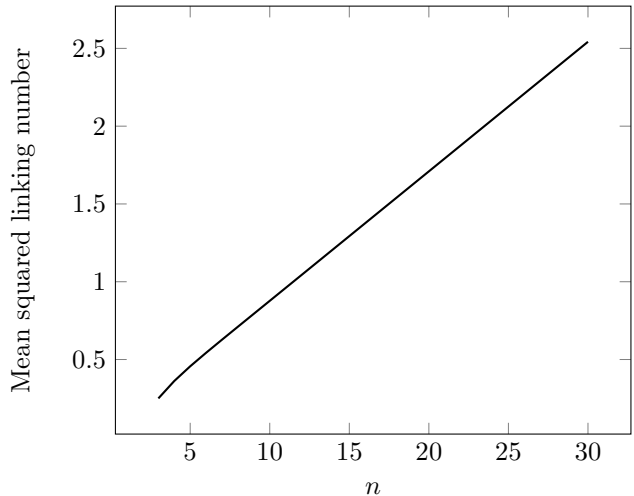

\section{Links in random book embeddings of $K_6$}

In this section, we consider the special case of random book embeddings of
$K_6$. Rowland has studied all possible topological types of
book embeddings of $K_6$, showing that the set of non-trivial knots and
links that appear are the trefoil knot, figure-eight knot, the Hopf link,
and the Solomon's link \cite{Rowland}. Any two-component link in $K_6$ must
consist
of two disjoint $3$-cycles, and every $3$-cycle is necessarily monotonic.
Moreover, the trivial link has linking number $0$, the Hopf link has
linking number $\pm 1$, and the Solomon's link (shown in Figure
\ref{fig:ascentlinkingnumberexample}) has linking number $\pm 2$.
Hence, we can utilize Theorems \ref{thm:cycle-interior-edge-distribution}
and \ref{thm:meansquaredlinking} and in the case that $n=3$ to determine
the probabilities of each type of link occuring in a random book embedding.

We separately consider the cases when the
number of interior edges in the $3$-cycles is $i=1$, $2$, and $3$ as in
Figure \ref{fig:k6interior}, and determine the probability of each type of
link occuring in each case. We can then combine with the counts in Table
\ref{tab:interior-edges-number-distribution-for-2ngons} to compute the overall
probability that a randomly selected two-component link is either
trivial, a Hopf link, or a Solomon's link.

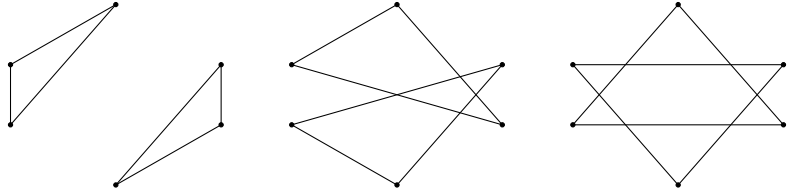
\begin{figure} \centering
\begin{tikzpicture}[scale=0.4]
	\draw[fill=black] (0,0) circle (2pt);
	\draw[fill=black] (-3.5,2) circle (2pt);
	\draw[fill=black] (3.5,2) circle (2pt);
	\draw[fill=black] (-3.5,4) circle (2pt);
	\draw[fill=black] (3.5,4) circle (2pt);
	\draw[fill=black] (0,6) circle (2pt);

	\draw (0,0)--(3.5,2);
	\draw (0,0)--(3.5,4);
	\draw (3.5,2)--(3.5,4);
	\draw (0,6)--(-3.5,2);
	\draw (0,6)--(-3.5,4);
	\draw (-3.5,2)--(-3.5,4);
\end{tikzpicture}
\hspace{0.25in}
\begin{tikzpicture}[scale=0.4]
	\draw[fill=black] (0,0) circle (2pt);
	\draw[fill=black] (-3.5,2) circle (2pt);
	\draw[fill=black] (3.5,2) circle (2pt);
	\draw[fill=black] (-3.5,4) circle (2pt);
	\draw[fill=black] (3.5,4) circle (2pt);
	\draw[fill=black] (0,6) circle (2pt);

	\draw (0,0)--(-3.5,2);
	\draw (0,0)--(3.5,4);
	\draw (-3.5,2)--(3.5,4);
	\draw (0,6)--(3.5,2);
	\draw (0,6)--(-3.5,4);
	\draw (3.5,2)--(-3.5,4);
\end{tikzpicture}
\hspace{0.25in}
\begin{tikzpicture}[scale=0.4]
	\draw[fill=black] (0,0) circle (2pt);
	\draw[fill=black] (-3.5,2) circle (2pt);
	\draw[fill=black] (3.5,2) circle (2pt);
	\draw[fill=black] (-3.5,4) circle (2pt);
	\draw[fill=black] (3.5,4) circle (2pt);
	\draw[fill=black] (0,6) circle (2pt);

	\draw (0,0)--(-3.5,4);
	\draw (0,0)--(3.5,4);
	\draw (-3.5,4)--(3.5,4);
	\draw (0,6)--(-3.5,2);
	\draw (0,6)--(3.5,2);
	\draw (-3.5,2)--(3.5,2);
\end{tikzpicture}
\caption{Projections of two $3$-cycles in $K_6$ with $i=1$ (left), $i=2$
(middle), and $i=3$ (right) interior edges.}
\label{fig:k6interior} \end{figure}

When $i=1$, it is evident that since the projection of the two cycles has
no crossings, then the two-component link is trivial.

When $i=2$, Table \ref{tab:Eulerianlinkingnumberdistributionbycrossing}
implies that the probability that the two cycles are the Hopf link is
$p_2=\frac{1}{3}$, and the probability that the two cycles are the trivial
link is $1-p_2=\frac{2}{3}$.

When $i=3$, Table \ref{tab:Eulerianlinkingnumberdistributionbycrossing}
implies that the probability that the two cycles form the Solomon's link is
$q_3=\frac{1}{60}$, the probability that the two cycles form the Hopf link is
$p_3 =\frac{13}{30}$, and the probability that the two cycles form the trivial
link is $1-p_3-q_3 = \frac{11}{20}$.

Table \ref{tab:interior-edges-number-distribution-for-2ngons} details the
frequency with each the 10 cycles in $K_6$ have $1$, $2$, or $3$ interior
edges. From this, we determine that the probability that a randomly chosen
pair of disjoint $3$-cycles in a random book embedding of $K_6$ is trivial is
\begin{equation*}
	\frac{1}{10} \left( 3\cdot 1 + 6\cdot \frac{2}{3}+1 \cdot
		\frac{11}{20}\right) = \frac{151}{200}.
\end{equation*}
Similarly, the probability that a randomly chosen pair of disjoint $3$-cycles
in a random book embedding of $K_6$ is the Hopf link is
\begin{equation*}
	\frac{1}{10} \left( 3 \cdot 0 + 6 \cdot \frac{1}{3} + 1 \cdot
		\frac{13}{30}\right) = \frac{73}{300}.
\end{equation*}
Finally, the probability that a randomly chosen pair of disjoint $3$-cycles
in a random book embedding of $K_6$ is the Solomon's link is
\begin{equation*}
	\frac{1}{10} \left( 3 \cdot 0 + 6 \cdot 0 + 1 \cdot
		\frac{1}{60}\right) = \frac{1}{600}.
\end{equation*}

Since $K_6$ contains 10 distinct disjoint pairs of $3$-cycles, this implies
that in a random book embedding of $K_6$, the expected number of trivial links
is $\frac{151}{20}$, the expected number of Hopf links is $\frac{73}{30}$,
and the expected number of Solomon's links is $\frac{1}{60}$. It is a
classical result in spatial graph theory that every embedding of $K_6$
contains at least one non-trivial link \cite{ConwayGordon}. In a random book
embedding of $K_6$, the expected number of non-trivial links is
$\frac{49}{20}$, with nearly all of the non-trivial links represented by
Hopf links.

\section{Acknowledgments}

The authors would like to thank the
National Science Foundation for supporting this work. This research was
partially supported by National Science Foundation Grant DMS-1852132.

In addition, the authors would like to thank the Department of Mathematics at
Rose-Hulman Institute of
Technology for their hospitality and for hosting the Rose-Hulman Institute of
Technology Mathematics Research Experience for Undergraduates, where most of
this work was completed.

\bibliographystyle{plain}
\bibliography{bibliography}

\end{document}